\documentclass[preprint,review,10pt]{elsarticle}

\usepackage{amsmath,amssymb,amsfonts,amsthm}
\usepackage{hyperref}
\usepackage{graphicx}
\usepackage{mathrsfs}

\usepackage{color}
\usepackage{cases}
\usepackage{graphicx}
\usepackage{multirow}
\usepackage{subfigure}
\usepackage{epstopdf}
\usepackage{epsfig}
\usepackage{float}
\newtheorem{assumption}{Assumption}
\newtheorem{corollary}{Corollary}
\newtheorem{lemma}{Lemma}
\newtheorem{remark}{Remark}
\newtheorem{theorem}{Theorem}
\newtheorem{definition}{Definition}

\let\mr=\mathrm

\begin{document}

\begin{frontmatter}



\title{Supercloseness of the DDG method for a singularly perturbed convection diffusion problem on Shishkin mesh\tnoteref{funding} }

\author[label1] {Xiaoqi Ma\fnref{cor1}}
\author[label1] {Jin Zhang\corref{cor2}}
\author[label1] {Xinyi Feng \fnref{cor3}}
\author[label1] {Chunxiao Zhang\fnref{cor4}}
\cortext[cor1] {Corresponding email: jinzhangalex@sdnu.edu.cn}
\fntext[cor2] {Others: xiaoqiMa@hotmail.com }
\address[label1]{School of Mathematics and Statistics, Shandong Normal University, Jinan 250014, China}

\begin{abstract}
This paper investigates the supercloseness of a singularly perturbed convection diffusion problem using the direct discontinuous Galerkin (DDG) method on a Shishkin mesh. The main technical difficulties lie in controlling the diffusion term inside the layer, the convection term outside the layer, and the inter element jump term caused by the discontinuity of the numerical solution. The main idea is to design a new composite interpolation, in which a global projection is used outside the layer to satisfy the interface conditions determined by the selection of numerical flux, thereby eliminating or controlling the troublesome terms on the unit interface; and inside the layer, Gau{\ss} Lobatto projection is used to improve the convergence order of the diffusion term. On the basis of that, by selecting appropriate parameters in the numerical flux,  we obtain the supercloseness result of almost $k+1$ order under an energy norm.
Numerical experiments support our main theoretical conclusion.
\end{abstract}

\begin{keyword}
Convection diffusion \sep DDG \sep Shishkin mesh \sep Supercloseness 

\end{keyword}

\end{frontmatter}


%
%
%
\section{Introduction}
In this paper, we study a direct discontinuous Galerkin (DDG) method to solve the following singularly perturbed two-point boundary value problem,
\begin{equation}\label{eq:SPP-1d}
\begin{aligned}
&-\epsilon w''(x)+a(x)w'(x)+b(x)w(x)=f(x),\quad x \in\Theta:=(0,1),\\
&w(0)=0, \quad w(1)=0,
\end{aligned}
\end{equation}
in which $0<\epsilon \ll 1$ is a perturbation parameter. Then we assume that $a$, $b$ and $f$ are sufficiently smooth functions with  $a(x)\ge \alpha>0$ on $\overline{\Theta}$ and  
\begin{equation}\label{eq:SPP-condition-1}
b-\frac{1}{2}a'\ge \gamma>0\quad \forall x \in \overline{\Theta},
\end{equation}
where $\alpha$ and $\gamma$ are fixed positive constants. These assumptions ensure that the boundary value problem \eqref{eq:SPP-1d} has a unique solution $w\in H_{0}^{1}(\Theta)\cap H^{2}(\Theta)$ for each $f\in L^{2}(\Theta)$; see \cite{Sty1Tob2:2003-motified}. Due to the parameter $\epsilon$ can be small enough, the solution $u$ usually produces a layer at $x = 1$ of width $\mathcal{O}(\epsilon\ln (1/\epsilon))$. At this point, numerical solutions obtained using the classical finite element method typically generate significant numerical oscillations in the whole region, see \citep[Section 6.1]{Sty1Sty2:2018-C}. 
 In order to correctly analyze these layers without increasing the scale of solving discrete problems, layer adapted meshes, which use a general mesh outside the layer  and a refined mesh inside the layer, have emerged.  On the other hand, due to the general interest in only global solutions, it is urgent to design stabilization methods to eliminate oscillations, regardless of the degree of mesh refinement.

 The current popular stability methods include streamline upwind/Petrov Galerkin (SUPG) method \cite{Miz1Hug2:1985-motified, Bur1:2010-C, Che1Xu2:2005-motified, Hug1Liu2:1979-F}, least squares finite element method \cite{Hug1Fra2Hul3:1989-motified}, bubble function stabilization \cite{Bre1Fra2:1998-F, Bre1Hug2:1999-motified, Bre1Rus2:1994-C}, local projection stabilization \cite{Bra1Bur2:2006-L, Mat1Skr2:2007--motified}, continuous internal penalty method \cite{Bur1:2005--motified, Bur1Ern2:2007-C}, the discontinuous Galerkin (DG) method and so on. Among them, DG method is a finite element method that utilizes a completely discontinuous piecewise polynomial space to solve numerical solutions and test functions. 
One attractive feature of this method is the flexibility provided by the combination of local approximation space and appropriate numerical flux across cell interfaces. Given this unique advantage, 
since it was first proposed by Reed and Hill in 1973 for the neutron transport equation,  many authors have extended it to various partial differential equations, see \cite{Car1Hop2:2011-A, Ces1Rhe2:2023-motified, Coc1Shu2:1998-motified, Hu1Hua2:2011-motified} for more details. 

However, most DG methods have a common feature of introducting  auxiliary variables in their formulation, such as the local discontinuous Galerkin (LDG) method \cite{Coc1Shu2:1998-motified}, which greatly increases the computational cost in numerical methods. To avoid this drawback, Liu and Yan \cite{Liu1Yan2：2008-motified} proposed a direct discontinuous Galerkin (DDG) method in 2009, which is based on the direct weak formulations for solutions and on appropriate numerical fluxes.
Compared with the LDG method, this method does not introduce new variables, so it has the advantages of easier formulation and efficient computation of numerical solutions.
Furthermore, the authors in \cite{Liu1Yan2：2008-motified} presented a general numerical flux formula for the derivative of the solution at the element interface, and introduced the concept of compatibility to identify suitable numerical fluxes.  Later, Liu et al. studied a more effective set of parameter selection in DDG numerical flux \cite{Hua1Liu2Yi3:2012-R}, demonstrating the optimal $L^{2}$ error estimate of the DDG method for solving convection diffusion problems in \cite{Liu1:2015-O}. 
It is worth noting that the DDG method degenerates into the classical interior penalty discontinuous Galerkin method by using  lower order piecewise constants and linear approximations \cite{Arn1:motified-1982, Whe:motified-1978}. And for higher order approximations ($k\ge 2$), the DDG method exhibits considerable advantages. For example, compared to the symmetric internal penalty discontinuous Galerkin method \cite{C2012-Discontinuous}, it does not require significant penalty parameters to maintain the stability of the scheme; Compared with the nonsymmetric interior penalty discontinuous Galerkin method \cite{Zha1Ma2:2021-S}, its computational costs will be reduced and etc, see \cite{Cao1Liu2Zha3:2017-S} for more details.

On the other hand, for the convective term, traditional DG methods usually use purely upwind numerical fluxes. 
However, for complex systems, it is difficult to construct pure upwind fluxes because they require accurate eigenstructures of the flux Jacobian matrix.
 In \cite{Men1Shu2Wu3:2016-O}, Meng et al. used a more general upwind-biased numerical flux, which may not always be monotonic, but can be more easily constructed, providing a good approximation for smooth solutions. On the basis of that, Ma and Stynes \cite{Ma1Sty2:2020-motified} analyzed a singularly perturbed problem using the DDG method and obtained an optimal error estimate of almost $k$ order. Inspired by \cite{Ma1Sty2:2020-motified}, we study the supercloseness of a singularly perturbed convection diffusion problem by using the DDG method.
In this process, the numerical flux in \cite{Ma1Sty2:2020-motified} is corrected, and the concept of compatibility is also introduced as a criterion for selecting appropriate numerical flux to ensure the stability and corresponding error estimation of the DDG method.


The ``supercloseness" here represent that the numerical solution is closer to the interpolation of the exact solution, rather than the exact solution itself.
The main technical difficulties of the DDG method for supercloseness analysis lie in the construction of special interpolation functions and the selection of parameters in numerical flux.  In \cite{Men1Shu2Wu3:2016-O}, Meng et al. constructed a suitable global projection based on the definition of upwind-biased flux, which can eliminate boundary terms between elements, thereby obtaining the desired supercloseness results. 
In this paper, we design a composite interpolation for the solution $u$. More specifically, Gau{\ss} Lobatto projection is used within the layer, and the global projection provided by \cite{Men1Shu2Wu3:2016-O} is used outside the layer. Then by choosing some specific values of the parameters in numerical flux, we obtain the supercloseness of almost $k+1$ order in the corresponding norm.
 Finally, a numerical experiment confirm the theoretical result.

The basic framework of this paper is organized as follows: In Section 2, we provide {\emph a priori} information of the solution, and introduce a layer-adapted mesh applicable to the continuous problem. In Section 3, the DDG method is presented and the concept of compatibility was given. Then, by defining a new composite interpolation, uniform convergence and supercloseness results are obtained in an energy norm in Section 4 and 5, respectively. Finally, our theoretical result can be validated through a  numerical experiment.


In this paper, let $k\ge 1$ be a fixed integer. Assume that $C$ is a generic positive constant and is independent of the perturbation parameter $\epsilon$ and the mesh parameter $N$. The value of $C$ varies in different situations.

%
%
%

\section{Regularity of the solution and Shishkin mesh}\label{sec:regularity,mesh,method}
\subsection{A-$prior$ information of the solution }\label{sec:regularity}
Now we give a-$prior$ estimates of the solution, which is crucial for the following error estimates.
\begin{theorem}\label{Thm:regularity}
Suppose that the condition \eqref{eq:SPP-condition-1} holds and the functions $a, b, f$ are sufficiently smooth. Then the solution $w$ of \eqref{eq:SPP-1d} admits the decomposition
\begin{equation*}\label{eq:decomposition}
w=S+E,
\end{equation*}
where  $S$ and $E$ meet $LS=f$ and $LE=0$, respectively. Moreover, for $0\le i\le k+1$, there is
\begin{equation}\label{eq:regularity}
\begin{aligned}
&|S^{(i)}(x)|\le C, \\
&|E^{(i)}(x)|\le C \epsilon^{-i}\exp\left(-\frac{ \alpha(1-x) }{ \epsilon }  \right).
\end{aligned}
\end{equation}
Note that here $k$ depends on the regularity of the coefficients, and when $a, b, f\in C^{\infty}(\Theta)$, \eqref{eq:regularity} holds. 
\end{theorem}
\begin{proof}
According to \citep[Lemma 1.9]{Roo1Sty2Tob3:2008-R}, the proof of this theorem can be easily obtained.
\end{proof}

\subsection{Shishkin mesh }
Assuming that $\Theta_{N}=\{x_{j} \in [0, 1] : j=0, 1, 2, \ldots, N\}$ is set of mesh points and 
\begin{center}
$\mathcal{T}_{N}$=$\{I_{j}=(x_{j-1},x_{j}): j=1, 2,  \ldots, N\}$
\end{center}
is a partition of the domain $\Theta$. Here $h_{j}=x_{j}-x_{j-1}$ is the length of the element $I_{j}$ and $I$ is a general interval. Define $\Delta h_{j}=\min\{h_{j}, h_{j+1}\}$ for $j=1, 2, \cdots, N$, $h_{0}=h_{1}$ and $h_{N+1} := h_{N}$.

To better resolve the layer appearing in the solution of problem \eqref{eq:SPP-1d}, we use a piecewise uniform--Shishkin mesh \cite{Shi1:1998-G}. First, $\Theta$ is divided into a  coarse-mesh interval $[0,1-\tau_{t}]$ and a fine-mesh interval $[1-\tau_{t},1]$, in which the transition point $1-\tau_{t}=1-\frac{\sigma \epsilon}{\alpha}\ln N$ meets $\tau_{t}\le 1/2$ and $\sigma\ge k+1$. Then we assume that $N\in\mathbb{N}$ is divisible by $2$ and $N\ge 4$. After dividing $[0, 1-\tau_{t}]$ and $[1-\tau_{t}, 1]$ into $N/2$ equal subintervals, the mesh can be denoted by
\begin{equation}\label{eq:Shishkin mesh-Roos}
x_{j}=
\left\{
\begin{split}
& \frac{2(1-\tau_{t})}{N}j\quad &&\text{$j=0, 1, 2, \ldots,  N/2$},\\
&1-\tau_{t}+\frac{2\tau_{t}}{N}(j-\frac{N}{2})\quad &&\text{$j=N/2+1, N/2+2, \ldots, N$}.
\end{split}
\right.
\end{equation}
Note that $x_{N/2}=1-\tau_{t}$.
\begin{assumption}\label{assumption}
In this article, we suppose that $$\epsilon\le CN^{-1}.$$ This is not a restriction in practice.
\end{assumption}
\section{The DDG method}
\subsection{The DDG method}
For $I_{j}, j=1, 2, \cdots, N$, let $m$ be a nonnegative interger, then the broken Sobolev space of order $m$ can be defined as
 \begin{equation*}
H^{m}(\Theta, \mathcal{T}_N)=\{v\in L^{2}({\Theta}):
 \text{$v|_{I_{j}}\in H^{m}(I_{j})$, $\forall j=1, 2, \ldots, N$} \}.
\end{equation*}
And then we provide the norm and seminorm 
$$\Vert v \Vert^{2}_{m, \mathcal{T}_N}=\sum_{j=1}^{N} \Vert v \Vert^{2}_{m, I_{j}},\quad  \vert v \vert^{2}_{m, \mathcal{T}_N}=\sum_{j=1}^{N}\vert v \vert^{2}_{m, I_{j}}$$
where $\Vert \cdot \Vert_{m, I_{j}}$ and $\vert \cdot \vert_{m, I_{j}}$ are the norm and semi-norm in $H^{m}(I_{j})$. Moreover,  the $L^{2}(I)$-norm and the $L^{2}(I)$-inner product are represented by $\Vert \cdot \Vert_{I}$ and $(\cdot,\cdot)_I$. Then on Shishkin mesh, we define the $k$-degree discontinuous finite elment space as
\begin{equation*}\label{eq:VN}
V_{h}=\{w\in L^{2}({\Theta}):
 w|_{I_{j}}\in \mathbb{P}_k(I_{j}), \forall j=1, 2, \ldots, N\},
\end{equation*}
where $\mathbb{P}_k(I_{j})$ is the space of real-valued polynomials with degree at most $k$ on $I_{j}$.

 First, define the jump and average at the interior nodes $x_{j}$ by
$$[v]_{j}=v(x_j^{+})-v(x_j^{-}),\quad\{v\}_{j}=\frac{1}{2}(v(x_{j}^{+})+v(x_{j}^{-})),\quad \forall j=1, 2, \ldots, N-1,$$
for all functions $v\in H^{1}(\Theta, \mathcal{T}_{N})$. Here $v(x_j^{+})$=$\lim\limits_{x\to {x_{j}^{+}}}v(x)$ and $v(x_{j}^{-})$=$\lim\limits_{x\to {x_j^{-}}}v(x)$.
Furthermore, at the boundary points $x_{0}$ and $x_{N}$, we set
$$[v]_{0}=v(x_{0}^{+}),\quad \{v\}_{0}=v(x_{0}^{+}),\quad [v]_{N}=-v(x_{N}^{-}),\quad \{v\}_{N}=v(x_{N}^{-}).$$

Now the weak form of problem \eqref{eq:SPP-1d} is to find $u_{h}\in V_{h}$ such that 
\begin{equation}\label{eq:weak form-1d}
B(w_{h}, v_{h})=F(v_{h}),\quad \forall v_{h}\in V_{h},
\end{equation}
where $B(w_{h}, v_{h})=B_{1}(w_{h}, v_{h})+B_{2}(w_{h}, v_{h})+B_{3}(w_{h}, v_{h})$ and
\begin{align}
&B_{1}(w_{h}, v_{h})=\sum_{j=1}^{N} \int_{I_j}\epsilon w_{h}'v_{h}'\mr{d}x+\epsilon\sum_{j=0}^{N}(\widehat{w_{h}}(x_j)[v_{h}]_{j}+[w_{h}]_{j}\{v_{h}'\}_{j}),\nonumber\\
&B_2(w_{h}, v_{h})=-\sum_{j=1}^{N} \int_{I_j}a(x)w_{h}v_{h}'\mr{d}x-\sum_{j=1}^{N}\int_{I_{j}}a'(x)w_{h}v_{h}\mr{d}x-\sum_{j=0}^{N}a(x_{j})\widetilde{w_{h}}(x_{j})[v_{h}]_{j},\label{eq:B-2}\\
&B_3(w_{h}, v_{h})=\sum_{j=1}^{N} \int_{I_{j}}b(x)w_{h}v_{h}\mr{d}x,\nonumber\\
&F(v_{h})=\sum_{j=1}^{N} \int_{I_{j}}fv_{h}\mr{d}x.\nonumber
\end{align}
Here, we use a numerical flux $\widehat{w_{h}}(x_{j})$ to replace the values of $w_{h}'(x_{j}), j=0, 1, \cdots, N$. In addition, we use another numerical flux $\widetilde{w_{h}}(x_{j})$ instead of the value of $w$ at $x_{j}$.

According to \citep[Section 2.2]{Liu1Yan2：2008-motified}, we construct the numerical flux $\widehat{w_{h}}(x_{j})$ of
\begin{equation}\label{eq:A-1}
\widehat{w_{h}}(x_{j})=\left\{
\begin{aligned}
&\frac{\beta_{0, j}}{\Delta h_{j}}[w_{h}]_{j}+\{w'_{h}\}_{j}+\beta_{1}\Delta h_{j}[w''_{h}]_{j},\quad &&\text{for $j=1, 2, \cdots, N-1$},\\
&w'_{h}(x_{0}),\quad &&\text{for $j=0$},\\
&w'_{h}(x_{N}),\quad &&\text{for $j=N$},
\end{aligned}
\right.
\end{equation}
where $\beta_{0, j} (j=1, 2, \cdots, N-1)$ and $\beta_{1}$ are the positive coefficients selected to ensure stability of the scheme and enhance accuracy. 
Specifically, when $k=1$, we take $\beta_{1}=0$, and the value of $\beta_{0, j} (j=1, 2, \cdots, N-1)$ in different situations will be given in the following lemma.

For $\widetilde{w_{h}}(x_{j})$ in \eqref{eq:B-2}, we use the so-called upwind-biased flux \citep[Section 2.2.1]{Men1Shu2Wu3:2016-O}, rather than the purely upwind flux. More specifically, we choose
\begin{equation}\label{eq:A-2}
\widetilde{w_{h}}(x_{j})=\left\{
\begin{aligned}
&\theta w_{h}(x_{j}^{-})+(1-\theta) w_{h}(x_{j}^{+}),\quad&& \text{$j=1, 2, \cdots, N-1$},\\
&0,\quad&& \text{ $j=0$},\\
&w_{h}(x_{N}^{-}),\quad&& \text{$j=N$},
\end{aligned}
\right.
\end{equation}
in which we choose $\frac{1}{2}\le \theta \le 1$ as $a(x) > 0$. 
\subsection{The admissibility of DDG method}
In order to guarantee stability, and more importantly, to measure the goodness of the selection of $\beta_{0, j} (j=1, \cdots, N-1)$ and $\beta_{1}$, we will present the concept of admissibility of numerical flux according to \cite{Liu1Yan2:2010-motified1}.
\begin{definition}\label{eq: K-1}
\citep[Definition 2.1]{Liu1Yan2:2010-motified1} If there exist constants $\mu_{1}\in (0, 1)$ and $\mu_{2}\in (0, 1]$ such that for $w_{h}\in V_{h}$
$$\mu_{1}\sum_{j=1}^{N}\int_{I_{j}}(w'_{h})^{2}\mr{d}x+\sum_{j=0}^{N}(\widehat{w_{h}}[u_{h}]_{j}+[w_{h}]_{j}\{w_{h}\}_{j})\ge \mu_{2}\sum_{j=0}^{N}\frac{[w_{h}]_{j}^{2}}{\Delta h_{j}},$$
 then the numerical flux $\widehat{w_{h}}$ of \eqref{eq:A-1} is said to be admissible.
\end{definition}
\begin{lemma}\label{LLLLL}
\citep[Theorem 2.1]{Liu1Yan2：2008-motified} Assume that there are constants $\mu_{1}\in (0, 1)$ and $\mu_{2}\in (0, 1]$ such that
\begin{equation*}\label{eq:A-3}
\beta_{0, j}\ge \mu_{2}+\frac{1}{\mu_{1}}M(k, \beta_{1}),
\end{equation*}
where $M(k, \beta_{1}): =\max\limits_{w\in \mathbb{P}_{k}(I_{j})}\frac{\sum_{j=0}^{N}\Delta h_{j}(\{w'\}_{j}+\frac{\beta_{1}}{2}\Delta h_{j}[w'']_{j})^{2}}{\sum_{j=1}^{N}\int_{I_{j}}(w')^{2}\mr{d}x}$. Then for any given degree $k\ge 1$ of piecewise polynomials, the numerical flux \eqref{eq:A-1} is admissible.
\end{lemma}
\begin{proof}
The specific proof can be directly referenced in \citep[Theorem 2.1]{Liu1Yan2：2008-motified}.
\end{proof}
\begin{corollary}
We can always select $\beta_{0, j}$ to make the flux \eqref{eq:A-1} admissible.
\end{corollary}
\begin{proof}
For $\forall w\in V_{h}$, by using the inverse inequality \citep[Theorem 3.2.6]{ciarlet2002finite}, 
\begin{equation*}
\begin{aligned}
\{w'\}_{j}^{2}\le C(\Vert w'\Vert^{2}_{L^{\infty}(I_{j})}+\Vert w'\Vert^{2}_{L^{ \infty}(I_{j+1})})\le C(\Delta h_{j})^{-1}(\Vert w'\Vert_{I_{j}}^{2}+\Vert w'\Vert_{I_{j+1}}^{2}).
\end{aligned}
\end{equation*}
In the same way, there is
\begin{equation*}
\begin{aligned}
(\Delta h_{j})^{2}[w'']_{j}^{2}\le C(\Delta h_{j})^{2}(\Vert w''\Vert^{2}_{L^{ \infty}(I_{j})}+\Vert w''\Vert_{L^{\infty}(I_{j+1})}^{2})\le C(\Delta h_{j})^{-1}(\Vert w'\Vert_{I_{j}}^{2}+\Vert w'\Vert^{2}_{I_{j+1}}).
\end{aligned}
\end{equation*}
Then, it is easy to obtain
\begin{equation*}
\begin{aligned}
\sum_{j=0}^{N}\Delta h_{j}(\{w'\}_{j}+\beta_{1}\Delta h_{j}[w'']_{j})^{2}
&\le C\sum_{j=0}^{N}\Delta h_{j}(\{w'\}_{j}^{2}+(\beta_{1}\Delta h_{j}[w'']_{j})^{2})\\
&\le C\sum_{j=0}^{N}\Delta h_{j}(\Delta h_{j})^{-1}(1+\beta_{1}^{2})(\Vert w'\Vert_{I_{j}}^{2}+\Vert w'\Vert^{2}_{I_{j+1}})\\
&\le C(1+\beta_{1}^{2})\sum_{j=0}^{N}\int_{I_{j}}(w')^{2}\mr{d}x\\
&\le C(k)\sum_{j=0}^{N}\int_{I_{j}}(w')^{2}\mr{d}x,
\end{aligned}
\end{equation*}
where $C(k)$ is a positive constant related to the degree $k$ of piecewise polynomials. Therefore, there exists a constant $C(k)$ to satisfy
$$\frac{\sum_{j=0}^{N}\Delta h_{j}(\{w'\}_{j}+\frac{\beta_{1}}{2}\Delta h_{j}[w'']_{j})^{2}}{\sum_{j=1}^{N}\int_{I_{j}}(w')^{2}\mr{d}x}\le C(k).$$ 
At this point, we can choose $\beta_{0, j} (j=0, 1, \cdots, N)$ by calculating $C(k)\frac{1}{\mu_{1}}+\mu_{2}$ to make the numerical flux \eqref{eq:A-1} admissibility.
\end{proof}
When $\beta_{1} = 0$, the DDG method shall degenerate into  the classical penalty method, in which $\beta_{0, j} (j=0, 1, \cdots, N)$ of sufficient size are required to stabilize the method. Now we will provide an alternative method to the lemma \ref{LLLLL} to calculate the appropriate value of $\beta_{0, j} (j=0, 1, \cdots, N)$ according to \citep[Lemma 2.2]{Ma1Sty2:2020-motified}. 
\begin{lemma}\label{PPPPPP}
For a fixed $k\ge 1$, \eqref{eq:A-1} with $\beta_{1} =0$ is admissible if
\begin{equation*}\label{PPPP-1}
\beta_{0}\ge\mu_{2}+\frac{1}{4\mu_{1}}\lambda_{max}(H^{-\frac{1}{2}}OH^{-\frac{1}{2}}), 
\end{equation*}
where $H$ is the Hilbert matrix $H= (\frac{1}{m+l-1})$ of size $k$ and $O$ is a $k\times k$ matrix with each entry to be $1$.
\end{lemma}
\begin{proof}
According to the theories in \citep[Lemma 2.2]{Ma1Sty2:2020-motified}, Lemma \ref{PPPPPP} can be easily obtained.
\end{proof}
From this result we can provide a possible option for $\beta_{0, j}$ when $k=1$. For instance, we take $\mu_{2}=1$ and $\mu_{1}=1/2$, and let $\beta_{0, j}$ to be an integer as
$$\beta_{0, j} =\left[\frac{1}{2}\lambda_{max}(H^{-\frac{1}{2}}OH^{-\frac{1}{2}})\right]+1,$$
in which $[y]$ is the smallest integer bigger than or equal to $y$. Then for $k\ge 1$, there is $\beta_{0, j}=[k^{2}/2]+1$.
\vspace{-0.5cm}
\subsection{Galerkin orthogonality and stability of the DDG method}
\begin{lemma}
Assume that the solution $w$ of \eqref{eq:SPP-1d} satisfy $w\in H^{k+1}(\Theta)$, then $B(\cdot,\cdot)$ given in \eqref{eq:weak form-1d} has the following Galerkin orthogonality,
$$B(w-w_{h}, v)=0,\quad\text{for all $v \in V_{h}$},$$
where $w_{h}$ is the solution of \eqref{eq:weak form-1d}.
\end{lemma}
\begin{proof}
By utilizing the similar arguments in \citep[Lemma 2.4]{Ma1Sty2:2020-motified}, we can derive this lemma easily.
\end {proof}
Then  from $B(\cdot,\cdot)$,  the energy norm is denoted by 
\begin{equation}\label{eq:energy norm}
\Vert v \Vert^{2}_{E}:=\epsilon \sum_{j=1}^{N} \Vert v' \Vert^{2}_{I_{j}}+\sum_{j=1}^{N}\gamma \Vert v \Vert^{2}_{I_j}+\epsilon\sum_{j=0}^{N}\frac{\beta_{0, j}}{\Delta h_{j}}[v]_{j}^{2},\quad\text{for $\forall v \in V_{h}$}.
\end{equation}
Here for obtaining the desired supercloseness result, we take $\beta_ {0, j}$ as
\begin{equation}\label{eq:A-4}
\beta_{0, j}=\left\{
\begin{aligned}
&\beta_{1}^{2}\epsilon^{-1}N^{-1},\quad \text{for $j=0, 1, \cdots, N/2-1$},\\
&\beta_{1}^{2},\quad \text{for $j= N/2$},\\
&\beta_{1}^{2}N,\quad \text{for $j= N/2+1, \cdots, N$}.
\end{aligned}
\right.
\end{equation}
where $\beta_{1}$ is a positive constant.
\begin{lemma}
Assume that the numerical flux is admissible and the condition \eqref{eq:SPP-condition-1} holds. Then, there is
\begin{equation}\label{eq:energy norm 1}
B(v_{h}, v_{h})\ge\Vert v_{h} \Vert^{2}_{E},\quad\forall v_{h}\in V_{h},
\end{equation}
where $B(\cdot,\cdot)$ is defined in \eqref{eq:weak form-1d} and $\Vert\cdot\Vert_{E}$ is defined as \eqref{eq:energy norm}.
\end{lemma}
\begin{proof}
The specific derivation can be found in \citep[Lemma 2.3]{Ma1Sty2:2020-motified}.
\end{proof}
%
%
%
\section{Interpolation and convergence analysis}
\subsection{Interpolation}
Now we design a new interpolation operator used for our convergence analysis. Set
\begin{equation}\label{eq:pi u}
(\pi w)|_{I_{j}}=
\left\{
\begin{split}
&(P^{(\theta)} w)|_{I_{j}}\quad &&\text{if $j=1, \ldots, N/2$},\\
& (G_{k} w)|_{I_{j}},\quad &&\text{if $ j=N/2+1, \ldots, N$},
\end{split}
\right.
\end{equation}
where $P^{(\theta)} w \in  V_{h}$ is a global projection of $w$ in the finite element space, and $G_{k} w$ is the Lagrange interpolation of $w$ at $k+1$ Gau{\ss} Lobatto points. Blow, we give the definitions and correspending properties of these projections.

When $k\ge 1$,  Gau{\ss} Radau projection $(P_{h}w)\vert_{I_{j}}\in \mathbb{P}_{k}(I_{j})$ is  denoted as (see \cite{Che1:2021-O}): For $j=1, 2, 3, \cdots, N$, 
\begin{equation}\label{eq:J-1}
\begin{aligned}
&\int_{I_j}(P_{h}w)v\mr{d}x=\int_{I_j}wv\mr{d}x,\quad \forall v\in \mathbb{P}_{k-1}(I_{j}),\\
&(P_{h}w)(x_{j}^{-})=w(x_{j}^{-}).
\end{aligned}
\end{equation}

For $w\in H^{1}(\mathcal{T}_{N})$, the projection
$P^{(\theta)} w$ is denoted as the element of $V_{h}$ that satisfies (see \citep[Section 2.4.1]{Men1Shu2Wu3:2016-O})
\begin{align}
&\int_{I_{j}}(P^{(\theta)}w)v\mr{d}x = \int_{I_{j}}wv\mr{d}x,\quad\forall v \in \mathbb{P}_{k-1}(I_{j}),\label{eq: C-1}\\
&\widetilde{P^{(\theta)}w}(x_{j})= \widetilde{w}(x_{j}),\quad \text{for $j=1, 2, \ldots, N-1$},\label{eq: C-2}\\
&(P^{(\theta)}w)(x_{N}^{-})=w(x_{N}^{-}).\label{eq: C-3}
\end{align}
Here from the upwind-biased numerical flux \eqref{eq:A-2}, we define $\widetilde{w} :=\theta w(x_{j}^{-}) + (1-\theta)  w(x_{j}^{+})$ with $\frac{1}{2}\le\theta\le 1$ for any $w\in H^{1}(\mathcal{T}_{N})$. In particular, when $\theta = 1$, it is obvious that $P^{(\theta)}w(x_{j}^{-}) = w(x_{j}^{-}), j=1, 2, \cdots, N$ and the projection $P^{(\theta)}w$ reduces to the standard Gau{\ss} Radau projection $P_{h} w$, see \cite{Men1Shu2Wu3:2016-O} for more details.

Then on $I_{j}=[x_{j-1}, x_{j}]$, let the Gau{\ss} Lobatto points $x_{j-1}=z_{0} <z_{1} <\cdots <z_{k} = x_{j}$, where $z_{1}, z_{2}, \cdots, z_{k-1}$ are zeros of the derivative of the Legendre polynomial whose degree $k\ge 1$ on $I_{j}$ \cite{Ma1Zha2:2023-S}.  Assume that $(G_{k} w)|_{I_{j}}$ for $j=1, \ldots ,N$ is the Lagrange interpolation using the points $\{z_{s}\}_{s=0}^{k}$ as the interpolation points, then we have 
\begin{equation}\label{eq:z v}
|(w'-(G_{k}w)', v')_{I_{j}}|\le Ch_{j}^{k+1}|w|_{k+2, {I_{j}}}|v|_{1, I_{j}},\quad\forall v \in \mathbb{P}_{k}(I_{j}),
\end{equation}
with $w\in H^{k+2}(I_{j})$.
\begin{remark}\label{special}
Now we briefly introduce the situation at $x_{\frac{N}{2}}$. Firstly, on the left side at the point $x_{\frac{N}{2}}$, by using \eqref{eq:pi u} and \eqref{eq: C-3}, one obtains
\begin{equation}\label{eq: Kk-11}
(w-P^{(\theta)} w)(x_{\frac{N}{2}}^{-})=0.
\end{equation}
On the other hand, on the right side at the point $x_{\frac{N}{2}}$, by the definition of Gau{\ss} Lobatto projection, 
\begin{equation}\label{eq: Kk-12}
(w-G_{k} w)(x_{\frac{N}{2}}^{+})=0.
\end{equation}
Then, combined \eqref{eq:pi u}, \eqref{eq: Kk-11}, \eqref{eq: Kk-12} and the definition of jump $[\cdot]$,  we have
\begin{equation*}
[(w-\pi w)]_{\frac{N}{2}}=0.
\end{equation*}
\end{remark}
\subsection{Interpolation error}
Now we will present existence and approximation properties of $P^{(\theta)} w$.
\begin{lemma}\label{FFFFFFF}
\citep[Lemma 2.6]{Men1Shu2Wu3:2016-O}  Suppose that $w \in H^{k+1}(\mathcal{T}_{N})$ is sufficiently smooth. Then, there exists a unique $P^{(\theta)} w$ satisfying the conditions \eqref{eq: C-1}, \eqref{eq: C-2} and \eqref{eq: C-3}. Furthermore, one has
\begin{align}
&\Vert w-P^{(\theta)}w \Vert_{L^{\infty}(I_j)}\le C(\theta) h^{k+1}\Vert w \Vert_{W^{k+1, \infty}(\mathcal{T}_{N})},\label{eq: interpolation-theory}\\
&\Vert w-P^{(\theta)}w \Vert_{I_{j}} \le C(\theta) h^{k+\frac{3}{2}}\Vert w \Vert_{W^{k+1, \infty}(\mathcal{T}_{N})}, \label{eq:interpolation-theory-1}
\end{align}
where $h = \max\limits_{j} h_{j}, j=1, 2, \cdots, N$, $C = C(\theta)$ is independent of $\epsilon$, the element $I_{j}$ and $h$.
\end{lemma}
\begin{proof}
For $w\in H^{1}(\mathcal{T}_{N})$, let $P_{h}w$ and $P^{(\theta)}w$ be denoted as \eqref{eq:J-1} and \eqref{eq: C-1}. Denote $P^{(\theta)}w - w = P^{(\theta)}w - P_{h}w + P_{h}w - w := E + P_{h}w - w$. Now we only prove the existence and uniqueness of $E$. 
First, combined \eqref{eq:J-1}, \eqref{eq: C-1}, \eqref{eq: C-2} and \eqref{eq: C-3}, we obtain
\begin{align}
&\int_{I_{j}}Ev_{h}\mr{d}x = 0, \quad \forall v_{h}\in \mathbb{P}_{k-1}(I_{j}),\label{SSSS-1}\\
&\widetilde{E} = (1 - \theta)(w - P_{h}w)(x_{j}^{+})\quad \text{at $x_{j},\quad j = 1, \cdots, N-1$},\label{SSSS-2}\\
&E(x_{j}^{-})= 0,\quad j = N\label{SSSS-3}. 
\end{align}
From $E\in V_{h}$, the restriction of $E$ to each $I_{j}$, denoted by $E_{j}$, can be expressed as $E_{j}(x)=\sum\limits_{l=0}^{k}c_{j, l}P_{j, l}(x)=\sum\limits_{l=0}^{k}c_{j, l}P_{l}(\hat{x})$, in which $P_{j, l}(x)= P_{l}\left(\frac{2\left(x-x_{j-\frac{1}{2}}\right)}{h_{j}}\right)=P_{l}(\hat{x})$. Here $P_{l}(\hat{x})$ is the $l$th-order Legendre polynomials that are orthogonal on $\bar{\Theta}$. 
Then from the orthogonality of the Legendre polynomials and  \eqref{SSSS-1},
$$c_{j, l} = 0,\quad j = 1, 2, \cdots, N,\quad l = 0, 1, \cdots, k- 1.$$
Thus, $E_{j} (x)= c_{j, k}P_{k}(\hat{x})$. Then noticing that $P_{k}(\pm 1) = (\pm 1)^{k}$, from \eqref{SSSS-2} and \eqref{SSSS-3}, we have the following algebra system of linear equations
\begin{equation*}
\mathbb{A}_{N\times N} =\left[
\begin{array}{cccccc}
\theta &(1-\theta)(-1)^{k} &0&\cdots &\cdots &0\\
0 &\theta &(1-\theta)(-1)^{k}&0&\cdots&0\\
0&0&\theta &(1-\theta)(-1)^{k}&0&0\\
\vdots&\vdots&\qquad\ddots&\quad\ddots&&\vdots\\
\vdots&\vdots&&\ddots&\ddots&0\\
0&0&&&\theta &(1-\theta)(-1)^{k}\\
0&0&\cdots&\cdots&0&1
\end{array}
\right]
\end{equation*}
\begin{equation*}
{\vec{c}}_{N\times 1} =\left(\begin{array}{c}
c_{1, k}\\
c_{2, k}\\
\vdots\\
c_{N-1, k}\\
c_{N, k}
\end{array}
\right),
\qquad{\vec{b}}_{N\times 1}=
\left(
\begin{array}{c}
\eta_{1}\\
\eta_{2}\\
\vdots\\
\eta_{N-1}\\
0
\end{array}
\right)
\end{equation*}
denoted by $\mathbb{A}_{N\times N} {\vec{c}}_{N\times 1} = {\vec{b}}_{N\times 1}$. Here we set $\eta_{j}= (w-P_{h}w)(x_{j}^{+}), j=1, 2, \cdots, N-1$. 
It is evident that this $N\times N$ linear system can be decoupled and solved explicitly with $c_{N, k} = 0$ as a starting point, i.e., for $j = 1, \cdots, N- 1$. In other words, from $det(\mathbb{A}_{N\times N}) = \theta^{N-1}$ and $\theta > 1/2$, we can conclude that $\mathbb{A}_{N\times N}$ is invertible,
which establishes the existence and uniqueness of $E$, and further $P^{(\theta)}w$. 

Then, we will obtain the approximation error estimates
\eqref{eq: interpolation-theory} and \eqref{eq:interpolation-theory-1}. For $j = 1, 2, \cdots, N-1$,
$$|\eta_{j}|\le C\Vert w-P_{h}w\Vert_{L^{\infty}(I_{j})}\le Ch_{j}^{k+1}\Vert w\Vert_{W^{k+1, \infty}(I_{j})}\le Ch^{k+1}\Vert w\Vert_{W^{k+1, \infty}(\mathcal{T}_{N})}.$$
Furthermore, we can derive $\mathbb{A}^{-1}$ as follows
\begin{equation*}
\mathbb{A}^{-1}=\left[
\begin{array}{cccccc}
\frac{1}{\theta} &(-1)^{k+1}\frac{1-\theta}{\theta^{2}}&0&\cdots&\cdots&0\\
0&\frac{1}{\theta} &(-1)^{k+1}\frac{1-\theta}{\theta^{2}}&0&\cdots&0\\
0&0&\frac{1}{\theta} &(-1)^{k+1}\frac{1-\theta}{\theta^{2}}&0&0\\
\vdots&\vdots&&\ddots&\ddots&\vdots\\
0&0&&&\frac{1}{\theta} &(-1)^{k+1}\frac{1-\theta}{\theta^{2}}\\
0&0&&&&1
\end{array}
\right]
\end{equation*}
and obtain the bound of $c_{j, k} (j=1, 2, \cdots, N-1)$ as follows: For $j=1, 2, \cdots, N-2$ and $\theta\ge \frac{1}{2}$,
\begin{equation*}
\begin{aligned}
|c_{j, k}|&=\frac{1-\theta}{\theta}|(w-P_{h}w)(x^{+}_{j})|+(-1)^{k+1}\frac{(1-\theta)^{2}}{\theta^{2}}|(w-P_{h}w)(x^{+}_{j+1})|\\
&\le C(\theta)\Vert w-P_{h}w\Vert_{L^{\infty}(I_{j+1})}+C(\theta)\Vert w-P_{h}w\Vert_{L^{\infty}(I_{j+2})}\\
&\le C(\theta)h^{k+1}\Vert w\Vert_{W^{k+1, \infty}(\mathcal{T}_{N})}.
\end{aligned}
\end{equation*}
For $j=N-1$, one has $$|c_{N-1, k}|\le \frac{1-\theta}{\theta}\eta_{N-1}\le C(\theta)h^{k+1}\Vert w\Vert_{W^{k+1, \infty}(\mathcal{T}_{N})}.$$
Then, it is easy to get 
$$\Vert E\Vert_{L^{\infty}(I_{j})}=\Vert c_{j, k}P_{j, k}(x)\Vert_{L^{\infty}(I_{j})}=|c_{j, k}|\Vert P_{k}\Vert_{L^{\infty}(\bar{\Theta})}\le C(\theta)h^{k+1}\Vert w\Vert_{W^{k+1, \infty}(\mathcal{T}_{N})}$$
and
\begin{equation*}
\begin{aligned}
\Vert E\Vert_{L^{2}(I_{j})}=\left(\int_{I_{j}}E^{2}\mr{d}x\right)^{\frac{1}{2}}\le C(\theta)\left(h_{j}\Vert E\Vert^{2}_{L^{\infty}(I_{j})}\right)^{\frac{1}{2}}\le C(\theta)h^{k+\frac{3}{2}}\Vert w\Vert_{W^{k+1, \infty}(\mathcal{T}_{N})}.
\end{aligned}
\end{equation*}
Hence the error estimates \eqref{eq: interpolation-theory} and \eqref{eq:interpolation-theory-1} follow immediately. 
\end{proof}

Assume that $G_{k}w$ is Gau{\ss} Lobatto projection defined in \S 4.1,  by using the interpolation theories in \cite[Theorem 3.1.4]{ciarlet2002finite}, one has
\begin{equation}\label{eq:interpolation-theory}
\Vert w-G_{k}w \Vert_{W^{l, q}(I_{j})}\le C h_{j}^{k+1-l+1/q-1/p}\vert w \vert_{W^{k+1, p}(I_{j})},
\end{equation}
for $w\in W^{k+1, p}(I_{j})$, in which $l=0, 1$ and $1\le p, q\le \infty$.
\begin{lemma}\label{lem:interpolation-error}
Let Assumption \ref{assumption}, $\sigma \ge k+1$ hold and assume that  $\beta_{0, j} (j=0, 1, 2, \cdots, N)$ are given in \eqref{eq:A-4}. Then on Shishkin mesh \eqref{eq:Shishkin mesh-Roos}, 
\begin{align}
&\Vert S-P^{(\theta)}S \Vert_{[0, x_{N/2-1}]} +\Vert w-\pi w \Vert_{[0, 1-\tau_{t}]}\le CN^{-(k+1)},\label{eq:interpolation-error-11}\\
&\Vert E-G_{k}E \Vert_{[1-\tau_{t}, 1]} +\Vert w-\pi w \Vert_{[1-\tau_{t}, 1]}\le  C\epsilon^{\frac{1}{2}}(N^{-1}\ln N)^{k+1},\label{eq:interpolation-error-12}\\
&\Vert (w-\pi w)' \Vert_{[0, 1-\tau_{t}]}\le C\left(N^{-k}+\epsilon^{-\frac{1}{2}}N^{-\sigma}+N^{1-\sigma}\right), \label{eq:interpolation-error-13}\\
&\Vert w-\pi w \Vert_{L^{\infty}(I_j)}\le C N^{-(k+1)}, \quad j=1, 2, \ldots, \frac{N}{2},\label{eq:interpolation-error-14}\\
&\Vert w-\pi w \Vert_{L^{\infty}(I_j)}\le C (N^{-1}\ln N)^{k+1},\quad j=\frac{N}{2}+1, \ldots, N,\label{eq:interpolation-error-15}\\
&\Vert G_{k}w-w \Vert_{E, [0, 1-\tau_{t}]} \le C\left(\epsilon^{1/2}N^{-k}+N^{-(k+1)}+\epsilon^{1/2}N^{1-\sigma}\right),\label{eq:DG-energy-norm-1}\\
&\Vert w-P^{(\theta)}w \Vert_{E, [0, 1-\tau_{t}]} \le C(N^{-1}\ln N)^{k+\frac{1}{2}}
.\label{eq:DG-energy-norm-11}
\end{align}
\end{lemma}
\begin{proof}
According to \eqref{eq: interpolation-theory}, \eqref{eq:interpolation-theory-1} and \eqref{eq:interpolation-theory}, \eqref{eq:interpolation-error-11}-\eqref{eq:interpolation-error-15} can be proved simply. Below, we will only analyze \eqref{eq:DG-energy-norm-1} and \eqref {eq:DG-energy-norm-11}.

From the definition of the DDG norm \eqref{eq:energy norm} and the property of Gau{\ss} Lobatto projection, $[G_{k}w-w]_{j}=0, j=0, 1, \ldots, N/2$. Further, $$\Vert G_{k}w-w \Vert^{2}_{E,  [0, x_{N/2}]}=\epsilon \sum_{j=1}^{N/2} \Vert (G_{k}w-w)' \Vert^{2}_{I_j}+\sum_{j=1}^{N/2}\gamma \Vert G_{k}w-w \Vert^{2}_{I_j},$$ where  triangle inequalities, \eqref{eq:interpolation-theory} and the inverse inequality \citep[Theorem 3.2.6]{ciarlet2002finite}  yield
\begin{align*}
|\epsilon \sum_{j=1}^{N/2} \Vert (G_{k}w-w)' \Vert_{I_j}^2|&\le C\epsilon \sum_{j=1}^{N/2} \Vert (G_{k}S-S)' \Vert_{I_j}^2+C\epsilon \sum_{j=1}^{N/2} \Vert (G_{k}E-E)' \Vert_{I_j}^2\\
&\le  C\epsilon \sum_{j=1}^{N/2} h_{j}^{2k}\Vert S^{(k+1)} \Vert_{I_j}^2+ C\epsilon \sum_{j=1}^{N/2} (\Vert (G_{k}E)' \Vert_{I_j}^2+\Vert E' \Vert_{I_j}^2)\\
&\le  C\epsilon \sum_{j=1}^{N/2} h_{j}^{2k+1}\Vert S^{(k+1)} \Vert_{L^{\infty}(I_j)}^{2} + C\epsilon \sum_{j=1}^{N/2} N\Vert E \Vert_{L^{\infty}(I_j)}^{2} + CN^{-2\sigma}\\
&\le C \epsilon N^{-2k} + C \epsilon N^{2-2\sigma}+ CN^{-2\sigma}.
\end{align*}
And from \eqref{eq:interpolation-theory} and the triangle inequalities, 
$$| \sum_{j=1}^{N/2} \gamma \Vert G_{k}w-w \Vert_{I_j}^2|\le C \sum_{j=1}^{N/2} \Vert G_{k}S-S \Vert_{I_j}^{2}+C \sum_{j=1}^{N/2} \Vert G_{k}E-E \Vert_{I_j}^{2} \le CN^{-2(k+1)}.$$
So far, the proof of \eqref{eq:DG-energy-norm-1} is complete.

Then, by using \eqref{eq:energy norm}, 
\begin{align*}
 \Vert w-P^{(\theta)}w \Vert_{E, [0, 1-\tau_{t}]}^{2}&= \epsilon \sum_{j=1}^{N/2} \Vert (w-P^{(\theta)}w)' \Vert_{I_j}^{2}+ \sum_{j=1}^{N/2}\gamma \Vert w-P^{(\theta)}w \Vert_{I_{j}}^{2}\\
&+\epsilon\sum_{j=0}^{N/2}\frac{\beta_{0, j}}{\Delta h_{j}}[w-P^{(\theta)}w]_{j}^{2}\\
&=P_{1}+P_{2}+P_{3}.
\end{align*}
Thus, the problem is transformed into analyzing $P_{1}$, $P_{2}$, $P_{3}$.

For $P_{1}$, using \eqref{eq:regularity} and \eqref{eq:interpolation-theory-1}, there is
$$| \sum_{j=1}^{N/2} \epsilon\Vert (S-P^{(\theta)}S)' \Vert_{I_j}^{2}|\le C\sum_{j=1}^{N/2} h_{j}^{2k+1}\epsilon \Vert S^{(k+1)} \Vert_{L^{\infty}(\mathcal{T}_{N})}^{2} \le C\epsilon N^{-2k},$$
and from the triangle inequalities and the inverse inequality \citep[Theorem 3.2.6]{ciarlet2002finite},
\begin{align*}
|\sum_{j=1}^{N/2}\epsilon  \Vert (E-P^{(\theta)}E)' \Vert_{I_{j}}^{2}|&\le  C \sum_{j=1}^{N/2}\epsilon (\Vert E' \Vert_{I_{j}}^{2}+\Vert (P^{(\theta)}E)' \Vert_{I_{j}}^2)\\
&\le  C \sum_{j=1}^{N/2} \epsilon\int_{I_{j}}\epsilon^{-2}e^{-2 \alpha(1-x)/\epsilon}\mr{d}x+C \sum_{j=1}^{N/2}\epsilon h_{j}^{-1}\Vert P^{(\theta)}E \Vert_{L^{\infty}(I_{j})}^{2} \\
&\le  C\epsilon^{-1} \int_{0}^{x_{N/2}}e^{-2 \alpha(1-x)/\epsilon}\mr{d}x+C\epsilon N^{2}N^{-2\sigma}\\
&\le CN^{-2\sigma}+C\epsilon N^{2}N^{-2\sigma}.
\end{align*}

For $P_{2}$, according to \eqref{eq:interpolation-error-11}, it is obviously to obtain $|P_{2}| \le CN^{-2(k+1)}.$

Finally, from \eqref{eq:interpolation-error-14} and \eqref{eq:A-4}, there is 
\begin{equation*}
\begin{aligned}
|P_{3}| &\le C\sum_{j=0}^{N/2-1}\frac{\epsilon\beta_{0, j}}{\Delta h_{j}} \Vert w-\pi w \Vert_{L^{\infty}({I_{j}}\cup{I_{j+1}})}^{2}+\frac{\epsilon\beta_{0, N/2}}{\Delta h_{N/2}} [w-P^{(\theta)}w]^{2}_{N/2}\\
&\le C\sum_{j=0}^{N/2-1}\frac{\epsilon\beta_{0, j}}{\Delta h_{j}} \Vert w-\pi w \Vert_{L^{\infty}({I_{j}}\cup{I_{j+1}})}^{2}+C\frac{\epsilon\beta_{0, N/2}}{\Delta h_{N/2}} \Vert w-P^{(\theta)}w\Vert^{2}_{L^{\infty}(I_{N/2})}\\
&\le C\beta_{0, j}\epsilon N^{-2k}+ C\beta_{0, N/2}(N^{-1}\ln N)^{2k+1}\\
&\le C(N^{-1}\ln N)^{2k+1}.
\end{aligned}
\end{equation*}
So far, the proof of \eqref {eq:DG-energy-norm-11} is complete.
\end{proof}
\begin{theorem}\label{BBB}
Under the same hypotheses as in Lemma \ref{lem:interpolation-error}, and on the layer-adapted mesh with $\sigma \ge k + 1$, 
$$\Vert w- \pi w\Vert_{E}\le C(N^{-1}\ln N)^{k},$$
in which $w$ is the solution of \eqref{eq:SPP-1d}, and $\pi w$  is introduced in \eqref{eq:pi u}.
\end{theorem}
\begin{proof}
Using the triangle inequality and the definition of  $\pi w$ \eqref{eq:pi u}, we have
$$\Vert w- \pi w\Vert_{E}\le C\Vert w-P^{(\theta)}w\Vert_{E, [0, x_{N/2}]}+C\Vert w-G_{k}w\Vert_{E, [x_{N/2}, 1]},$$
where the estimation of $\Vert w-P^{(\theta)}w\Vert_{E, [0, x_{N/2}]}$  has been known from \eqref{eq:DG-energy-norm-11}, therefore, we only need to analyze $\Vert w-G_{k}w\Vert_{E, [x_{N/2}, 1]}$ below. 

From \eqref{eq:interpolation-error-12} and \eqref{eq:interpolation-error-15}, one has $\Vert w-G_{k}w\Vert_{E, [x_{N/2}, 1]}\le C(N^{-1}\ln N)^{k}$. Therefore, this proof is complete.
\end{proof}
\section{Supercloseness}
Introduce $\xi:=\pi w-w_h$ and $\eta:=\pi w-w$.
From \eqref{eq:energy norm 1} and the Galerkin orthogonality, 
\begin{equation}\label{eq:chi DG}
\begin{split}
&\Vert \xi \Vert_{E}^2 \le B(\xi,\xi)=B(\pi w-w+w-w_h,\xi)=B(\eta,\xi)\\
&=\sum_{j=1}^{N} \int_{ I_j }\epsilon \eta'\xi'\mr{d}x+\epsilon\sum_{j=0}^{N}\widehat{\eta}(x_j)[\xi]_{j}+\epsilon\sum_{j=0}^{N}[\eta]_{j}\{\xi'\}_{j}\\
&-\sum_{j=1}^{N} \int_{ I_j }a(x)\eta \xi'\mr{d}x-\sum_{j=0}^{N}a(x_{j})\widetilde{\eta}(x_j)[\xi]_{j}+\sum_{j=1}^{N} \int_{ I_{j} }(b-a')\eta \xi \mr{d}x\\
&=:\mathcal{M}_{1}+\mathcal{M}_{2}+\mathcal{M}_{3}+\mathcal{M}_{4}+\mathcal{M}_{5}+\mathcal{M}_{6}.
\end{split}
\end{equation}
Below, we will estimate the right-hand side of \eqref{eq:chi DG} item by item. Firstly, we decompose $\mathcal{M}_{1}$ as
$$\mathcal{M}_{1}=\sum_{j=1}^{N/2} \int_{I_{j}}\epsilon \eta'\xi'\mr{d}x+\sum_{j=N/2+1}^{N} \int_{ I_j }\epsilon \eta'\xi'\mr{d}x.$$
For $j=1, 2, \ldots, N/2$, from H\"{o}lder inequalities and \eqref{eq:interpolation-error-13},
\begin{equation}\label{eq:chi eta}
\begin{split}
|\sum_{j=1}^{N/2} \int_{I_{j}}\epsilon \eta'\xi'\mr{d}x|&\le C\left(\sum_{j=1}^{N/2}\epsilon \Vert \eta' \Vert_{I_{j}}^{2}\right)^{\frac{1}{2}}+\left(\sum_{j=1}^{N/2}\epsilon \Vert \xi' \Vert_{I_{j}}^{2}\right)^{\frac{1}{2}}\\
&\le C(\epsilon^{1/2}N^{-k}+\epsilon^{1/2}N^{1-\sigma}+N^{-\sigma})\Vert \xi \Vert_{E}.
\end{split}
\end{equation}
For $j=N/2+1, \ldots, N$, by H\"{o}lder inequalities and \eqref{eq:interpolation-theory}, 
\begin{equation}\label{eq:chi eta-2}
|\sum_{j=N/2+1}^{N} \int_{ I_j }\epsilon (S-G_{k}S)'\xi'\mr{d}x|\le C\epsilon^{k+1}N^{-k}(\ln N)^{k+\frac{1}{2}}\Vert \xi \Vert_{E}.
\end{equation}
In particular, according to \eqref{eq:z v}, we derive
\begin{equation}\label{eq:chi E-LE}
\begin{split}
&|\sum_{j=N/2+1}^{N} \int_{ I_j }\epsilon (E-G_{k}E)'\xi'\mr{d}x|\le C\sum_{j=N/2+1}^{N}\epsilon h_{j}^{k+1}\Vert E^{(k+2)} \Vert_{I_j}\Vert \xi' \Vert_{I_j}\\
&\le C\left(\sum_{j=N/2+1}^{N} \epsilon(\epsilon N^{-1}\ln N)^{2(k+1)}\Vert E^{(k+2)} \Vert_{I_j}^2\right)^{\frac{1}{2}}\left(\sum_{j=N/2+1}^{N} \epsilon \Vert \xi' \Vert_{I_j}^2\right)^{\frac{1}{2}}\\
&\le C\left(\epsilon^{-1}(N^{-1}\ln N)^{2(k+1)}\int_{x_{N/2}}^{1} e^{-2 \alpha(1-x)/\epsilon}\mr{d}x\right)^{\frac{1}{2}}\Vert \xi \Vert_{E}\\
&\le CN^{-(k+1)}(\ln N)^{k+1}\Vert \xi \Vert_{E}.
\end{split}
\end{equation}
Combining \eqref{eq:chi eta}, \eqref{eq:chi eta-2} and \eqref{eq:chi E-LE}, we prove
\begin{equation}\label{eq:mr I}
\mathcal{M}_{1}\le C\left(\epsilon^{1/2}N^{-k}+\epsilon^{1/2}N^{1-\sigma}+(N^{-1}\ln N)^{k+1}\right)\Vert \xi \Vert_{E}.
\end{equation}

Then according to the definition of numerical flux \eqref{eq:A-1}, $\mathcal{M}_{2}$ can be written as
\begin{equation*}\label{eq:C-1}
\begin{aligned}
\mathcal{M}_{2}=\epsilon\eta'(x_{0})[\xi]_{0}+\epsilon\eta'(x_{N})[\xi]_{N}+\sum_{j=1}^{N-1}\epsilon\left(\frac{\beta_{0, j}}{\Delta h_{j}}[\eta]_{j}+\{\eta'\}_{j}+\beta_{1}\Delta h_{j}[\eta'']_{j}\right)[\xi]_{j},
\end{aligned}
\end{equation*}
where from \eqref{eq:interpolation-theory} and \eqref{eq: interpolation-theory}, we have
\begin{equation*}
\begin{aligned}
|\epsilon\eta'(x_{0})[\xi]_{0}|\le C\epsilon \Vert\eta'\Vert_{L^{\infty}(I_{1})}\Vert\xi\Vert_{L^{\infty}(I_{1})}
\le C(\epsilon N^{\frac{1}{2}-k}+N^{\frac{1}{2}-\sigma})\Vert\xi\Vert_{E}
\end{aligned}
\end{equation*}
and in a similar method,
\begin{equation*}
\begin{aligned}
|\epsilon \eta'(x_{N})[\xi]_{N}|&\le C\epsilon^{\frac{1}{2}}\sqrt{\frac{h_{N}}{\beta_{0, j}}}\eta'(x_{N})\left(\sqrt{\frac{\epsilon\beta_{0, j}}{h_{N}}}[\xi]_{N}\right)\\
&\le C\epsilon^{\frac{1}{2}}\sqrt{\frac{h_{N}}{\beta_{0, j}}}\Vert\eta'\Vert_{L^{\infty}(I_{N})}\Vert\xi\Vert_{E}\\
&\le C\frac{1}{\sqrt{\beta_{0, j}}}(N^{-1}\ln N)^{k+\frac{1}{2}}\Vert\xi\Vert_{E}\\
&\le CN^{-(k+1)}(\ln N)^{k+\frac{1}{2}}\Vert\xi\Vert_{E}.
\end{aligned}
\end{equation*}
Note that $\beta_{0, N}=\beta_{1}^{2}N$. In addition, for $\sum\limits_{j=1}^{N-1}\frac{\epsilon\beta_{0, j}}{\Delta h_{j}}[\eta]_{j}[\xi]_{j}$, based on \eqref{eq:pi u}, Remark \ref{special} and Gau{\ss} Lobatto projection, there is $[\eta]_{j}=0$ for $j=N/2, N/2+1, \cdots, N$. Therefore,  we just analyze the situation for $j=1, \cdots, N/2-1$:
\begin{equation*}
\begin{aligned}
&\sum_{j=1}^{N/2-1}\frac{\epsilon\beta_{0, j}}{\Delta h_{j}}[\eta]_{j}[\xi]_{j}\\
&\le C\left(\sum_{j=1}^{N/2-1}\frac{\epsilon\beta_{0, j}}{\Delta h_{j}}[\eta]^{2}_{j}\right)^{\frac{1}{2}}\left(\sum_{j=1}^{N/2-1}\frac{\epsilon\beta_{0, j}}{\Delta h_{j}}[\xi]^{2}_{j}\right)^{\frac{1}{2}}\\
&\le C\left(\sum_{j=1}^{N/2-1}\frac{\epsilon\beta_{0, j}}{N^{-1}}\Vert\eta\Vert^{2}_{L^{\infty}(I_{j}\cup I_{j+1})}\right)^{\frac{1}{2}}\Vert\xi\Vert_{E}\\
&\le C\beta_{0, j}^{\frac{1}{2}}\epsilon^{\frac{1}{2}} N^{-k}\Vert\xi\Vert_{E}\\
&\le CN^{-(k+\frac{1}{2})}\Vert\xi\Vert_{E},
\end{aligned}
\end{equation*}
where \eqref{eq:interpolation-error-14}, $\beta_{0, j}=\beta_{1}^{2}\epsilon^{-1}N^{-1}$ for $j=0, 1, \cdots, N/2-1$ have been used. For $\sum\limits_{j=1}^{N-1}\epsilon\{\eta'\}_{j}[\xi]_{j}$, we first divide it into the following parts:
$$\sum_{j=1}^{N-1}\epsilon\{\eta'\}_{j}[\xi]_{j}=\sum_{j=1}^{N/2-1}\epsilon\{\eta'\}_{j}[\xi]_{j}+\epsilon\{\eta'\}_{N/2}[\xi]_{N/2}+\sum_{j=N/2+1}^{N-1}\epsilon\{\eta'\}_{j}[\xi]_{j},$$
where from H\"{o}lder inequalities, $\epsilon\le CN^{-1}$, $\sigma\ge k+1$, $\beta_{0, j}=\beta_{1}^{2}\epsilon^{-1}N^{-1}, j=1, 2, \cdots, N/2-1$, we have
\begin{equation}\label{eq: WWW-11}
\begin{aligned}
\sum_{j=1}^{N/2-1}\epsilon\{\eta'\}_{j}[\xi]_{j}
&\le C\epsilon\left(\sum_{j=1}^{N/2-1}\frac{\Delta h_{j}}{\epsilon\beta_{0, j}}\{\eta'\}_{j}^{2}\right)^{\frac{1}{2}}\left(\sum_{j=1}^{N/2-1}\frac{\epsilon\beta_{0, j}}{\Delta h_{j}}[\xi]_{j}^{2}\right)^{1/2}\\
&\le C\epsilon\left(\sum_{j=1}^{N/2-1}\frac{N^{-1}}{\epsilon\beta_{0, j}}\Vert\eta'\Vert^{2}_{L^{\infty}(I_{j}\cup I_{j+1})}\right)^{1/2}\Vert\xi\Vert_{E}\\
&\le C\epsilon\left(N\frac{N^{-1}}{\epsilon\beta_{0, j}}(N^{-2k}+\epsilon^{-2}N^{-2\sigma})\right)^{1/2}\Vert\xi\Vert_{E}\\
&\le C\frac{1}{\sqrt{\beta_{0, j}}}(\epsilon^{1/2}N^{-k}+\epsilon^{-1/2}N^{-\sigma})\Vert\xi\Vert_{E}\\
&\le C(\epsilon N^{\frac{1}{2}-k}+N^{\frac{1}{2}-\sigma})\Vert\xi\Vert_{E}
\end{aligned}
\end{equation}
and in a similar way,
\begin{equation*}
\begin{aligned}
\sum_{j=N/2+1}^{N-1}\epsilon\{\eta'\}_{j}[\xi]_{j}
&\le C\left(\sum_{j=N/2+1}^{N-1}\frac{\epsilon \Delta h_{j}}{\beta_{0, j}}\{\eta'\}_{j}^{2}\right)^{\frac{1}{2}}\left(\sum_{j=N/2+1}^{N-1}\frac{\epsilon\beta_{0, j}}{\Delta h_{j}}[\xi]_{j}^{2}\right)^{\frac{1}{2}}\\
&\le C\left(\sum_{j=N/2+1}^{N-1}\frac{\epsilon \Delta h_{j}}{\beta_{0, j}}\{\eta'\}_{j}^{2}\right)^{\frac{1}{2}}\Vert\xi\Vert_{E}\\
&\le C\left(N\frac{\epsilon^{2}N^{-1}\ln N}{\beta_{0, j}}\epsilon^{-2}(N^{-1}\ln N)^{2k}\right)^{\frac{1}{2}}\Vert\xi\Vert_{E}\\
&\le C\frac{1}{\sqrt{\beta_{0, j}}}N^{-k}(\ln N)^{k+\frac{1}{2}}\Vert\xi\Vert_{E}\\
&\le CN^{-(k+\frac{1}{2})}(\ln N)^{k+\frac{1}{2}}\Vert\xi\Vert_{E}.
\end{aligned}
\end{equation*}
Note that here $\Delta h_{j}=\min\{h_{j}, h_{j+1}\}$ and $\beta_{0, j}=\beta_{1}^{2}N$ for $j=N/2+1, N/2+2, \cdots, N$. When $j=N/2$, we derive
\begin{equation*}
\begin{aligned}
\epsilon\{\eta'\}_{N/2}[\xi]_{N/2}&\le C\epsilon\sqrt{\frac{\Delta h_{j}}{\epsilon\beta_{0, j}}}\{\eta'\}_{N/2}\left(\sqrt{\frac{\epsilon\beta_{0, j}}{\Delta h_{j}}}[\xi]_{N/2}\right)\\
&\le C\epsilon\sqrt{\frac{\Delta h_{j}}{\epsilon\beta_{0, j}}}\Vert\eta'\Vert_{L^{\infty}(I_{N/2}\cup I_{N/2+1})}\Vert\xi\Vert_{E}\\
&\le C\frac{1}{\sqrt{\beta_{0, j}}}\epsilon N^{-\frac{1}{2}}(\ln N)^{\frac{1}{2}}\left(N^{-k}+\epsilon^{-1}N^{-\sigma}\right)\Vert\xi\Vert_{E}\\
&\le C\frac{1}{\sqrt{\beta_{0, j}}}\left(\epsilon N^{-(k+\frac{1}{2})}(\ln N)^{\frac{1}{2}}+N^{-(\sigma+\frac{1}{2})}(\ln N)^{\frac{1}{2}}\right)\Vert\xi\Vert_{E}\\
&\le C\left(\epsilon N^{-(k+\frac{1}{2})}(\ln N)^{\frac{1}{2}}+N^{-(\sigma+\frac{1}{2})}(\ln N)^{\frac{1}{2}}\right)\Vert\xi\Vert_{E}\\
\end{aligned}
\end{equation*}
with $\beta_{0, j}=\beta_{1}^{2}$ for $j = N/2$. Then using the same method, we also decompose $\sum_{j=1}^{N-1}\epsilon\beta_{1}\Delta h_{j}[\eta'']_{j}[\xi]_{j}$ as
\begin{equation*}
\begin{aligned}
\sum_{j=1}^{N-1}\epsilon\beta_{1}\Delta h_{j}[\eta'']_{j}[\xi]_{j}
&\le \sum_{j=1}^{N/2}\epsilon\beta_{1}\Delta h_{j}[(S-P^{(\theta)}S)'']_{j}[\xi]_{j}\\
&+\sum_{j=1}^{N/2}\epsilon\beta_{1}\Delta h_{j}[(E-P^{(\theta)}E)'']_{j}[\xi]_{j}+\sum_{j=N/2+1}^{N-1}\epsilon\beta_{1}\Delta h_{j}[\eta'']_{j}[\xi]_{j},
\end{aligned}
\end{equation*}
where from \eqref{eq: interpolation-theory}, $\epsilon\le CN^{-1}$ and \eqref{eq:A-4},
\begin{equation*}
\begin{aligned}
\sum_{j=1}^{N/2}\epsilon\beta_{1}\Delta h_{j}[(S-P^{(\theta)}S)'']_{j}[\xi]_{j}&\le C\left(\sum_{j=1}^{N/2-1}\frac{\epsilon\beta_{1}^{2}(\Delta h_{j})^{3}}{\beta_{0, j}}\Vert(S-P^{(\theta)}S)''\Vert^{2}_{L^{\infty}(I_{j}\cup I_{j+1})}\right)^{\frac{1}{2}} \left(\sum_{j=1}^{N/2-1}\frac{\epsilon\beta_{0, j}}{\Delta h_{j}}[\xi]_{j}^{2}\right)^{\frac{1}{2}}\\
&+C\epsilon^{\frac{1}{2}}\beta_{1}h_{N/2+1}^{\frac{3}{2}}\beta_{0, j}^{-\frac{1}{2}}\Vert(S-P^{(\theta)}S)''\Vert_{L^{\infty}(I_{N/2}\cup I_{N/2+1})}\Vert\xi\Vert_{E}\\
&\le C\left( \sum_{j=1}^{N/2-1}\frac{\beta_{1}^{2}\epsilon N^{-3}}{\beta_{0, j}}h_{j}^{2(k-1)}\right)^{\frac{1}{2}}\Vert\xi\Vert_{E}+C\frac{\beta_{1}}{\sqrt{\beta_{0, j}}}\epsilon^{2}(N^{-1}\ln N)^{\frac{3}{2}}N^{1-k}\Vert\xi\Vert_{E}\\
&\le C\left(\frac{\beta_{1}}{\sqrt{\beta_{0, j}}}\epsilon^{\frac{1}{2}}N^{-k}+\frac{\beta_{1}}{\sqrt{\beta_{0, j}}}\epsilon^{2} N^{-(k+\frac{1}{2})}(\ln N)^{\frac{3}{2}}\right)\Vert\xi\Vert_{E}\\
&\le C\left(\epsilon N^{\frac{1}{2}-k}+\epsilon^{2} N^{-(k+\frac{1}{2})}(\ln N)^{\frac{3}{2}}\right)\Vert\xi\Vert_{E}\\
&\le C\epsilon N^{\frac{1}{2}-k}\Vert\xi\Vert_{E}.
\end{aligned}
\end{equation*}
In addition, due to the fact that $E''\in C[0, 1]$ and \eqref{eq:A-4}, one can derive
\begin{equation*}
\begin{aligned}
&\sum_{j=1}^{N/2}\epsilon\beta_{1}\Delta h_{j}[(E-P^{(\theta)}E)'']_{j}[\xi]_{j}\\&\le C\left(\sum_{j=1}^{N/2-1}\frac{\epsilon (\Delta h_{j})^{3}\beta_{1}^{2}}{\beta_{0, j}}[(E-P^{(\theta)}E)'']_{j}^{2}\right)^{\frac{1}{2}} \Vert\xi\Vert_{E}\\
&+C\epsilon^{\frac{1}{2}}\frac{\beta_{1}}{\sqrt{\beta_{0, j}}}h_{N/2+1}^{\frac{3}{2}}\Vert (E-P^{(\theta)}E)''\Vert_{L^{\infty}(I_{N/2}\cup I_{N/2+1})}\Vert\xi\Vert_{E}\\
&\le C\left(\sum_{j=1}^{N/2-1}\frac{\epsilon (\Delta h_{j})^{3}\beta_{1}^{2}}{\beta_{0, j}}(E''(x_{j}^{+})-(P^{(\theta)}E)''(x_{j}^{+})-E''(x_{j}^{-})+(P^{(\theta)}E)''(x_{j}^{-}))\right)^{\frac{1}{2}}\Vert\xi\Vert_{E}\\
&+C\epsilon^{\frac{1}{2}}\frac{\beta_{1}}{\sqrt{\beta_{0, j}}}(\epsilon N^{-1}\ln N)^{\frac{3}{2}}\left(\epsilon^{-2}N^{-\sigma}+N^{2-\sigma}\right)\Vert\xi\Vert_{E}\\
&\le C\left(N\frac{\epsilon N^{-3}\beta_{1}^{2}}{\beta_{0, j}}N^{4-2\sigma}\right)^{\frac{1}{2}}\Vert\xi\Vert_{E}+C\frac{\beta_{1}}{\sqrt{\beta_{0, j}}}N^{-(\frac{3}{2}+\sigma)}(\ln N)^{\frac{3}{2}}\Vert\xi\Vert_{E}\\
&\le C\left(\frac{\beta_{1}}{\sqrt{\beta_{0, j}}}\epsilon^{\frac{1}{2}}N^{1-\sigma}+\frac{\beta_{1}}{\sqrt{\beta_{0, j}}}N^{-(\frac{3}{2}+\sigma)}(\ln N)^{\frac{3}{2}}\right)\Vert\xi\Vert_{E}\\
&\le C\left(\epsilon N^{\frac{3}{2}-\sigma}+N^{-(\frac{3}{2}+\sigma)}(\ln N)^{\frac{3}{2}}\right)\Vert\xi\Vert_{E}\\
\end{aligned}
\end{equation*}
Then in the same way, from $\beta_{0, j}=\beta_{1}^{2}N$ for $j=N/2+1, \cdots, N$, there is
\begin{equation*}
\begin{aligned}
\sum_{j=N/2+1}^{N-1}\epsilon\beta_{1}\Delta h_{j}[\eta'']_{j}[\xi]_{j}&\le 
C\left(\sum_{j=N/2+1}^{N-1}\frac{\epsilon (\Delta h_{j})^{3}\beta_{1}^{2}}{\beta_{0, j}}\Vert\eta''\Vert^{2}_{L^{\infty}(I_{j}\cup I_{j+1})}\right)^{1/2}\Vert\xi\Vert_{E}\\
&\le C\left(N\frac{\epsilon(\epsilon N^{-1}\ln N)^{3}\beta_{1}^{2}}{\beta_{0, j}}\epsilon^{-4}(N^{-1}\ln N)^{2(k-1)}\right)^{1/2}\Vert\xi\Vert_{E}\\
&\le C\frac{\beta_{1}}{\sqrt{\beta_{0, j}}}N^{-k}(\ln N)^{k+\frac{1}{2}}\Vert\xi\Vert_{E}\\
&\le CN^{-(k+\frac{1}{2})}(\ln N)^{k+\frac{1}{2}}\Vert\xi\Vert_{E}.
\end{aligned}
\end{equation*}
In summary, by some simple calculations, it is straightforward to obtain that
\begin{equation}\label{P-2}
\mathcal{M}_{2}\le CN^{-(k+\frac{1}{2})}(\ln N)^{k+\frac{1}{2}}\Vert\xi\Vert_{E}.
\end{equation}
%
%
%

Next we can divide  $\mathcal{M}_{3}$ as
$$\mathcal{M}_{3}=-\epsilon \sum_{j=0}^{N/2-1}[\eta]_{j}\{\xi'\}_{j}-\epsilon \sum_{j=N/2}^{N}[\eta]_{j}\{\xi'\}_{j}.$$
From \eqref{eq:pi u} and Remark \ref{special}, we only analyze $-\epsilon\sum\limits_{j=0}^{N/2-1}[\eta]_{j}\{\xi'\}_{j}$. Applying the inverse inequality and \eqref{eq:interpolation-error-14}, 
\begin{equation}\label{eq:chi'(x_i)[eta(x_i)]}
\begin{split}
&|-\epsilon \sum_{j=0}^{N/2-1}[\eta]_{j}\{\xi'\}_{j}|\\
&\le|\epsilon [\eta]_{0}\{\xi'\}_{0}|+|\epsilon \sum_{j=1}^{N/2-1}[\eta]_{j}\{\xi'\}_{j}|\\
&\le C \Vert \eta \Vert_{L^{\infty}(I_{1})}\epsilon \Vert \xi' \Vert_{L^{\infty}(I_{1})}+C \sum_{j=1}^{N/2-1}\Vert \eta \Vert_{L^{\infty}(I_{j} \cup {I_{j+1}})}\epsilon \Vert \xi' \Vert_{L^{\infty}(I_{j} \cup {I_{j+1}})}\\
&\le C \epsilon\Vert \eta \Vert_{L^{\infty}(I_{1})} N^{\frac{1}{2}}\Vert \xi' \Vert_{I_{1}}+C\epsilon N^{\frac{1}{2}} \sum_{j=1}^{N/2-1}\Vert \eta \Vert_{L^{\infty}(I_{j} \cup {I_{j+1}})} \Vert \xi' \Vert_{I_{j} \cup {I_{j+1}}}\\
&\le C\epsilon^{1/2}N^{\frac{1}{2}}N^{-(k+1)}\Vert \xi \Vert_{E}+C\epsilon^{1/2}N^{\frac{1}{2}}N^{-(k+\frac{1}{2})}\Vert \xi \Vert_{E}\\
&\le C(\epsilon^{1/2}N^{-(k+\frac{1}{2})}+\epsilon^{1/2}N^{-k})\Vert \xi \Vert_{E}\\
&\le CN^{-(k+\frac{1}{2})}\Vert \xi \Vert_{E}.
\end{split}
\end{equation}


We consider $\mathcal{M}_{4}$ and $\mathcal{M}_{5}$, which are divided into $1 \le j \le N/2$ and $N/2+1 \le j \le N$ for analysis. For $1 \le j \le N/2$, through \eqref{eq: C-1}, \eqref{eq: C-2}, \eqref{eq: C-3}, and \eqref{eq: Kk-12}, we derive $\eta(x_{N/2})=0$ and
\begin{equation*}\label{eq:1-N/2 V VI }
\begin{split}
&-\sum_{j=1}^{N/2} \int_{ I_j }a\eta \xi'\mr{d}x-\sum_{j=0}^{N/2}a(x_j)\widetilde{\eta}(x_{j})[\xi]_{j}\\
&=-\sum_{j=1}^{N/2} \int_{ I_j }(a-a(x_{j-\frac{1}{2}}))\eta \xi'\mr{d}x-\sum_{j=1}^{N/2} \int_{ I_j }a(x_{j-\frac{1}{2}})\eta \xi'\mr{d}x-\sum_{j=0}^{N/2}a(x_j)\widetilde{\eta}(x_{j})[\xi]_{j}\\
&=-\sum_{j=1}^{N/2} \int_{ I_j }(a-a(x_{j-\frac{1}{2}}))\eta \xi'\mr{d}x,
\end{split}
\end{equation*}
where $a(x_{j-\frac{1}{2}})$ is the value of $a(x)$ at the point $x_{j-\frac{1}{2}}$ on $[x_{j-1}, x_{j}]$. Then from the Lagrange mean value theorem, there exists $\lambda$ between $x$ and $x_{j-\frac{1}{2}}$ such that
$$a(x)-a(x_{j-\frac{1}{2}})=a'(\lambda)(x-x_{j-\frac{1}{2}}).$$
Here $a$ is a smooth function. Therefore, adopting the inverse inequality and \eqref{eq:interpolation-error-14}, 
\begin{equation*}\label{eq:1-N/2 ax }
\begin{split}
&|-\sum_{j=1}^{N/2} \int_{ I_j }(a(x)-a(x_{j-\frac{1}{2}}))\eta \xi'\mr{d}x|=|-\sum_{j=1}^{N/2} \int_{ I_j }a'(\lambda)(x-x_{j-\frac{1}{2}})\eta \xi'\mr{d}x|\\
&\le C \sum_{j=1}^{N/2}h_{j} \Vert \eta \Vert_{L^{\infty} (I_{j})}\Vert \xi' \Vert_{L^{1}{(I_{j})}}\le C \sum_{j=1}^{N/2}h_{j} N^{\frac{1}{2}}\Vert \eta \Vert_{L^{\infty} (I_{j})}\Vert \xi \Vert_{I_{j}} \\
&\le C \sum_{j=1}^{N/2}N^{-\frac{1}{2}} \Vert \eta \Vert_{L^{\infty} (I_{j})}\Vert \xi \Vert_{I_{j}}\\
&\le CN^{-(k+1)}\Vert \xi \Vert_{E}.
\end{split}
\end{equation*}

For $N/2+1\le j \le N$, recalling $[\eta(x_{j})]=0$ and $\eta(x_{j})=0$,  we will estimate
$$-\sum_{j=N/2+1}^{N} \int_{ I_j }a(x)\eta \xi'\mr{d}x,$$
where from the inverse inequality, \eqref {eq:interpolation-error-12} and H\"{o}lder inequalities,
\begin{equation*}\label{eq:N/2+1-N ax }
\begin{split}
&|-\sum_{j=N/2+1}^{N} \int_{ I_{j} }a(x)\eta \xi'\mr{d}x|\le C\left( \sum_{j=N/2+1}^{N} \Vert \eta \Vert_{I_{j}}^{2}\right)^{1/2}\left(\sum_{j=N/2+1}^{N}\Vert \xi' \Vert_{I_{j}}^{2}\right)^{1/2}\\
&\le C \epsilon^{\frac{1}{2}}(N^{-1}\ln N)^{k+1}\left(\sum_{j=N/2+1}^{N}\Vert \xi' \Vert_{I_{j}}^{2}\right)^{1/2}\\
&\le C(N^{-1}\ln N)^{k+1}\Vert\xi\Vert_{E}.
\end{split}
\end{equation*}
Summing up, we obatin
\begin{equation}\label{eq:V VI }
\mathcal{M}_{4}+\mathcal{M}_{5}\le C(N^{-1}\ln N)^{k+1}\Vert \xi \Vert_{E}.
\end{equation}

For $\mathcal{M}_{6}$, according to H\"{o}lder inequalities,  \eqref{eq:interpolation-error-11} and \eqref{eq:interpolation-error-12}, 
\begin{equation}\label{eq:VII VIII }
\mathcal{M}_{6}\le C\Vert \eta \Vert_{[0, 1]}\Vert \xi \Vert_{E} \le C(N^{-(k+1)}+\epsilon^{\frac{1}{2}}(N^{-1}\ln N)^{k+1})\Vert \xi \Vert_{E}.
\end{equation}

From \eqref{eq:mr I}, \eqref{P-2}, \eqref{eq:chi'(x_i)[eta(x_i)]}, \eqref{eq:V VI } and \eqref{eq:VII VIII }, there is
\begin{equation*}\label{eq:x DG }
\begin{split}
\Vert \xi \Vert_{E}^{2}=\mathcal{M}_{1}+\mathcal{M}_{2}+\mathcal{M}_{3}+\mathcal{M}_{4}+\mathcal{M}_{5}+\mathcal{M}_{6}\le CN^{-(k+\frac{1}{2})}(\ln N)^{k+\frac{1}{2}}\Vert \xi \Vert_{E},
\end{split}
\end{equation*}
which implies
\begin{equation}\label{eq:pi u-un DG}
\Vert \pi w-w_h \Vert_{E}\le CN^{-(k+\frac{1}{2})}(\ln N)^{k+\frac{1}{2}}.
\end{equation}

Now we are in a position to present  the main result.  
\begin{theorem}\label{the:main result}
Let Assumption \ref{assumption} hold and $\beta_{0, j}$ be defined in \eqref{eq:A-4}. On the layer-adapted mesh \eqref{eq:Shishkin mesh-Roos} and $\sigma\ge k+1$, one has
\begin{align*}\label{eq:u-uN}
 \Vert G_{k}w-w_{h} \Vert_{E}+ \Vert \pi w-w_{h} \Vert_{E}\le &C N^{-(k+\frac{1}{2})}(\ln N)^{k+\frac{1}{2}},
\end{align*}
where $G_{k}w$ is Gau{\ss} Lobatto projection, $w_h$ is the solution of \eqref{eq:weak form-1d} and $\pi w$ is the interpolation of $u$.
\end{theorem}
\begin{proof}
First, for $ \Vert G_{k}w-w_h \Vert_{E}$, through the triangle inequality, 
$$ \Vert G_{k}w-w_h \Vert_{E}\le  \Vert G_{k}w-\pi w\Vert_{E}+\Vert \pi w-w_h \Vert_{E}$$
According to \eqref{eq:pi u-un DG}, the estimate of $\Vert G_{k}w-\pi w\Vert_{E}$ has been obtained, thus, we just estimate  the bound of $ \Vert G_{k}w-\pi w \Vert_{E}$.

From the triangle inequality and the interpolation $\pi w$ \eqref{eq:pi u}, 
\begin{equation*}\label{eq:lku-piu DG }
\begin{split}
\Vert G_{k}w-\pi w \Vert_{E}&= \Vert G_{k}w-G_{k}w \Vert_{E, [1-\tau_{t}, 1]}+\Vert G_{k}w-P^{(\theta)}w \Vert_{E, [0, 1-\tau_{t}]}\\
&= \Vert G_{k}w-P^{(\theta)}w \Vert_{E, [0, 1-\tau_{t}]}\\
&\le \Vert G_{k}w-w \Vert_{E, [0, 1-\tau_{t}]}+\Vert w-P^{(\theta)}w \Vert_{E, [0, 1-\tau_{t}]}.
\end{split}
\end{equation*}
According to \eqref{eq:DG-energy-norm-1} and \eqref{eq:DG-energy-norm-11}, we have
$$\Vert G_{k}w-P^{(\theta)}w \Vert_{E, [0, 1-\tau_{t}]}\le C(N^{-1}\ln N)^{k+\frac{1}{2}}.$$
Combining with \eqref{eq:pi u-un DG}, we have completed the proof of this theorem.
\end{proof}
\begin{theorem}\label{the:main result1}
Let $\epsilon\le CN^{-2}$,  $\sigma\ge k+2$ hold and $\beta_{0, j}$ be defined as
\begin{equation}\label{eq:Arrr-4}
\beta_{0, j}=\left\{
\begin{aligned}
&\beta_{1}^{2},\quad&& \text{$j=0, 1, 2, \cdots, N/2-1$},\\
&\beta_{1}^{2}N^{-1},\quad &&\text{$j= N/2$},\\
&\beta_{1}^{2}N^{2},\quad&&\text{$j=N/2+1, \cdots, N$}.
\end{aligned}
\right.
\end{equation}
On the mesh \eqref{eq:Shishkin mesh-Roos}, one has
\begin{align*}\label{eq:u-uN}
 \Vert G_{k}w-w_{h} \Vert_{E}+ \Vert \pi w-w_{h} \Vert_{E}\le &C N^{-(k+1)}(\ln N)^{k+1},
\end{align*}
in which $G_{k}w$ is Gau{\ss} Lobatto projection, $w_h$ is the solution of \eqref{eq:weak form-1d} and $\pi w$ is the interpolation of the solution $w$ of \eqref{eq:SPP-1d}.
\end{theorem}
\begin{proof}
Using the values of parameter $\beta_{0, j}, j=0, 1, \cdots, N$ \eqref{eq:Arrr-4} we derive  
\begin{equation*}
\Vert \pi w-w_h \Vert_{E}\le CN^{-(k+1)}(\ln N)^{k+1},
\end{equation*}
where we use the method different from \eqref{eq: WWW-11} to estimate $\sum\limits_{j=1}^{N/2-1}\epsilon\{\eta'\}_{j}[\xi]_{j}$,
\begin{equation*}
\begin{aligned}
\sum_{j=1}^{N/2-1}\epsilon\{\eta'\}_{j}[\xi]_{j}&\le \sum_{j=1}^{N/2-1}\epsilon\{\eta'\}_{j}\Vert\xi\Vert_{L^{\infty}(I_{j})}\le C\sum_{j=1}^{N/2-1}\epsilon h^{-1/2}_{j}\{\eta'\}_{j}\Vert\xi\Vert_{I_j}\\&\le C\left(\sum_{j=1}^{N/2-1}\epsilon^{2} h^{-1}_{j}\{\eta'\}_{j}^{2}\right)^{1/2}\Vert\xi\Vert_{E}\\&\le C\left(\sum_{j=1}^{N/2-1}\epsilon^{2} h^{-1}_{j}\Vert\eta'\Vert_{L^{\infty}(I_{j}\cup I_{j+1})}^{2}\right)^{1/2}\Vert\xi\Vert_{E}\\
&\le C(\epsilon N^{1-k}+N^{1-\sigma})\Vert\xi\Vert_{E}\\
&\le CN^{-(k+1)}\Vert\xi\Vert_{E}.
\end{aligned}
\end{equation*}
Then, applying the method similar to Theorem \ref{the:main result}, one can easily obtain this theorem.
\end{proof}
\begin{remark}
To obtain the desired convergence result, we introduce
the choose of $\beta_{0, j}$ for $j=0, 1, 2, \ldots, N$.
\begin{itemize}
\item To reach the order of $k+\frac{1}{2}$, 
for $j= N/2+1, \cdots, N$, the following condition
$$\beta_{0, j}\ge C\max\{1, N, \beta_{1}^{2}N, N^{-1}\}$$ holds and for $j= 0, 1, \cdots, N/2-1$, there is
$$C\max\{\epsilon^{-1}N^{-1}, \beta_{1}^{2}\epsilon N, \beta_{1}^{2}\epsilon N^{2k+3-2\sigma}, \epsilon N, \epsilon^{-1}N^{2k+1-2\sigma}\}\le \beta_{0, j} \le C\epsilon^{-1}N^{-1}.$$ In addition, for $j=N/2$, 
$$C\max\{\epsilon^{2}, N^{2k-2\sigma}, \beta_{1}^{2}\epsilon^{4}, \beta_{1}^{2}N^{2(k-\sigma-1)}\}\le \beta_{0, j}\le C.$$
In this paper, we choose  $\beta_{0, j}= \beta_{1}^{2}$ for $j=0, 1, \cdots, N/2$, $\epsilon \le CN^{-1}$ and $\sigma \ge {k+1}$.
\item To get the supercloseness of order $k+1$, for $j= N/2+1, \cdots, N$ we choose $\beta_{0}\ge C\max\{N^{2}, \beta_{1}^{2}N^{2}\}$, and for $j=0, 1, \cdots, N/2-1$, 
$$C\max\{\beta_{1}^{2}\epsilon N^{2}, \beta_{1}^{2}\epsilon N^{2k+4-2\sigma}\}\le \beta_{0, j}\le C\epsilon^{-1}N^{-2}.$$ 
Furthermore, for $j=N/2$, 
$$C\max\{\epsilon^{2}N, N^{2k+1-2\sigma}, N^{-1}, \beta_{1}^{2}\epsilon^{4}N, \beta_{1}^{2}N^{2k-1-2\sigma}\}\le \beta_{0, j}\le CN^{-1}.$$
Here we take $\beta_{0, j}=\beta_{1}^{2}N^{-1}$ for $j=N/2$ and $\beta_{0, j}=\beta_{1}^{2}$ for $j=0, 1, \cdots, N/2-1$. Note that $\epsilon \le CN^{-2}$ and $\sigma \ge {k+2}$.
\end{itemize}
\end{remark}

\begin{remark}
In this paper, when $k=1$, $\beta_{1}=0$. For obtaining the desired result, we also introduce the choose of $\beta_{0, j}$ for $j=0, 1, 2, \ldots, N$.
\begin{itemize}
\item To reach the order of $k+\frac{1}{2}$, we choose $\beta_{0, j}\ge CN, j= N/2+1, \cdots, N$, and for $j= 0, 1, \cdots, N/2-1$, 
$$C\max\{\epsilon N, \epsilon^{-1}N^{2k+1-2\sigma}\}\le \beta_{0, j} \le C\epsilon^{-1}N^{-1}.$$ In addition, for $j=N/2$, 
$$C\max\{\epsilon^{2}, N^{2k-2\sigma}\}\le \beta_{0, j}\le C.$$
Here we take  $\beta_{0, j}= C\epsilon^{-1}N^{-1}$ for $j=0, 1, \cdots, N/2-1$, $\beta_{0, N/2}=C$, $\epsilon \le CN^{-1}$ and $\sigma \ge {k+1}$.
\item To get the supercloseness result of order $k+1$, first we choose $\beta_{0, j}\ge CN^{2}$ for $j= N/2+1, \cdots, N$, and 
$\beta_{0, j}\le C\epsilon^{-1}N^{-2}$ for $j=0, 1, \cdots, N/2-1$.
Furthermore, for $j=N/2$, 
$$C\max\{\epsilon^{2}N, N^{2k+1-2\sigma}\}\le \beta_{0, j}\le CN^{-1}.$$
In this paper, we can take $\beta_{0, j}= N^{-1}$ for $j=N/2$ and $\beta_{0, j}=2$ for $j=0, 1, \cdots, N/2-1$. Note that here $\epsilon \le CN^{-2}$ and $\sigma \ge {k+2}$.
\end{itemize}
\end{remark}
\section{Numerical experiments}
In this section, we shall support the theoretical conclusion by considering the following singularly perturbed problem.
\begin{equation}\label{eq:KK-2}
\left\{
\begin{aligned}
 &-\epsilon w''(x)+(3-x)w'(x)+w(x)=f(x),\quad x\in \Theta: = (0,1),\\
&w(0)=w(1)=0,
\end{aligned}
\right.
\end{equation}
where $f(x)$ is chosen such that
\begin{equation*}
w(x)=x(1-e^{2(1-x)/\epsilon})
\end{equation*}
is the exact solution of the \eqref{eq:KK-2}. 


First, in the numerical example, we choose $\epsilon= 10^{-8}, N=8, 16, \cdots, 256$, $k=2, 3, 4$. Then we set $\alpha=2$, $\theta=\frac{2}{3}$, $\beta_{1}=1/(2k^{2}+2k)$, $\sigma = k +2$, and define $\beta_{0, j}$ as \eqref{eq:Arrr-4}.

Furthermore, when $k=1$ and $\beta_{1}=0$, we take $\epsilon= 10^{-8}, N=8, 16, \cdots, 512$, set $\alpha=2$, $\theta=\frac{2}{3}$, $\sigma = k +2$, and define $\beta_{0, j}$ as 
\begin{equation*}\label{eq:Arrr-4}
\beta_{0, j}=\left\{
\begin{aligned}
&2,\quad&& \text{for $j=0, 1, \cdots, N/2-1$},\\
&N^{-1},\quad&& \text{for $j= N/2$},\\
&N^{2},\quad&&\text{for $j=N/2+1, \cdots, N$}.
\end{aligned}
\right.
\end{equation*}

Then, the corresponding convergence rate is defined by
$$p_{N}= \frac{\ln e^{N}-\ln e^{2N}}{\ln \frac{2\ln N}{\ln 2N}},$$
where $e^{N}= \Vert G_{k}w-w_{h}\Vert_{E}$ is a computation error. 

\begin{table}[h]
\caption{$\Vert G_{k}w-w_h\Vert_{E}$ in the case of $\epsilon=10^{-8}$}
\footnotesize
\begin{tabular*}{\textwidth}{@{\extracolsep{\fill}} c cccccc}
\cline{1-7}{}
            \multirow{2}{*}{ $N$ }   &\multicolumn{2}{c}{$k=1$} &\multicolumn{2}{c}{$k=2$}  &\multicolumn{2}{c}{$k=3$}  \\
\cline{2-7}&$e^{N}$&$p_{N}$&$e^{N}$&$p_{N}$&$e^{N}$&$p_{N}$\\
\cline{1-7}
             $8$    & 0.668E-1  &1.49    &0.200E-1  &2.24  &0.524E-2 &2.87  \\
             $16$   &0.311E-1  &1.84  &0.631E-2  &2.76 &0.120E-2 &3.61 \\
             $32$   &0.125E-1  &2.15  &0.161E-2 &3.25  &0.202E-3 &4.26 \\
             $64$  &0.460E-2  &2.42  &0.355E-3  &3.69  &0.278E-4 &4.74 \\
             $128$   &0.160E-2  &2.66  & 0.709E-4  &4.08  &0.351E-5 &3.82\\
             $256$   &0.538E-3  &2.88  &0.133E-4  &3.67  &0.735E-6 &-4.18 \\
\cline{1-7}
\end{tabular*}
\label{table:1}
\end{table}
%
As predicted by Theorem \ref{the:main result1}, the numerical results are shown in Table \ref{table:1}.
 and Figure \ref{HH-1}.
\section{Statements and Declarations}
\subsection{Funding}
This research is supported by National Natural Science Foundation of China (11771257) and Shandong Provincial Natural Science Foundation, China (ZR2021MA004).
\subsection{Data availability statement}
The authors confirm that the data supporting the findings of this study are available within the article and its supplementary materials.
\subsection{Conflict of interests}
The authors declare that they have no conflict of interest.
%









%
%

\end{document}